\documentclass{amsart}%
\usepackage[utf8]{inputenc}
\usepackage{amsmath, amssymb, amsthm, amsfonts}
\usepackage{tikz-cd}
\usepackage{amsmath}
\usepackage{amsfonts}
\usepackage{amssymb}
\usepackage{graphicx}%
\providecommand{\U}[1]{\protect\rule{.1in}{.1in}}
\newtheorem{theorem}{Theorem}[section]
\newtheorem{lemma}[theorem]{Lemma}
\newtheorem{proposition}[theorem]{Proposition}
\newtheorem{corollary}[theorem]{Corollary}
\theoremstyle{definition}
\newtheorem{definition}[theorem]{Definition}
\newtheorem{example}[theorem]{Example}
\newtheorem{remark}[theorem]{Remark}
\begin{document}
\title[SDFA-primary Submodules]{Generalized Square-Difference Factor Absorbing Submodules of Modules over
Commutative Rings}
\author{Violeta Leoreanu-Fotea$^* $}
\address{Faculty of Mathematics, Al.I. Cuza University, Bd. Carol I, No. 11, 700506
Ia\c{s}i, Romania.}
\email{foteavioleta@gmail.com}
\thanks{*Corresponding author: foteavioleta@gmail.com} 
\author{Ece Yetkin Celikel}
\address{Department of Software Engineering, Hasan Kalyoncu University, Gaziantep, Turkiye}
\email{ece.celikel@hku.edu.tr, yetkinece@gmail.com}
\author{Tarik Arabaci}
\address{Department of Basic Science, Faculty of Engineering and Architecture, İstanbul Geli\c{s}im University, Istanbul, T\"urkiye}
\email{tarabaci@gelisim.edu.tr}
\author{Unsal Tekir}
\address{Department of Mathematics, Marmara University, Istanbul, T\"urkiye}
\email{utekir@marmara.edu.tr}
\maketitle

\begin{abstract}
In this paper, we introduce and study the class of generalized
square-difference factor absorbing (gsdf-absorbing) submodules of modules over
commutative rings. We provide various characterizations and properties of
gsdf-absorbing submodules and examine the behavior of this class of submodules
in some module extensions, including localization, homomorphic images, direct
products, idealization, and amalgamation. We also characterize all
gsdf-absorbing submodules of the $%
\mathbb{Z}
$-module $%
\mathbb{Z}
$. Several examples are provided to illustrate the results and to distinguish
this class from related notions.

\end{abstract}

\section{Introduction}

All rings considered in this paper are commutative with identity, and all
modules are assumed to be unital. Let $R$ be such a ring with $1\neq0$, and
let $M$ be an $R$-module. The notions of prime and primary ideals are
fundamental tools in commutative algebra, especially in connection with
primary decomposition \cite{AM1969}. Their extensions to submodules have been
extensively studied, leading to the well-developed theory of prime and primary
submodules \cite{M1992,SYZ1994}.

A particularly useful setting is provided by multiplication modules,
introduced in \cite{ES1988}, where each submodule is of the form $IM$ for some
ideal $I$ of $R$. This framework creates a strong bridge between
ideal-theoretic and module-theoretic properties, allowing one to transfer
results between the two contexts.

Several generalizations of prime and primary ideals have been proposed, many
of them relying on absorbing-type conditions. Among these, the concepts of
$2$-absorbing and $2$-absorbing primary ideals, introduced by Badawi and
collaborators \cite{B2007,BTY}, have received considerable attention. Recall
that for an ideal $I$ of $R$, its radical $\sqrt{I}$ consists of all elements
$u\in R$ such that $u^{n}\in I$ for some positive integer $n$. A proper ideal
$I$ of $R$ is called $2$-absorbing if whenever $u,v,w\in R$ with $uvw\in I$
implies that one of $uv$, $uw$, or $vw$ lies in $I$. A $2$-absorbing primary
ideal satisfies the condition: $uvw\in I$ implies that $uv\in I$ or
$uw\in\sqrt{I}$ or $vw\in\sqrt{I}$. A different direction was recently
initiated through the notion of square-difference factor absorbing ideals
\cite{AB2024}. A proper ideal $I$ of $R$ is said to be sdf-absorbing if
$u^{2}-v^{2}\in I$ for some $0\neq u,v\in$ $R$ forces $u+v\in I$ or $u-v\in I$.
Square-difference
factor absorbing primary ideals are studied in \cite {KhashanCelikel}.

It is known that a proper submodule $N$ of an $R$-module $M$ is primary if
whenever $u\in R$, $m\in M$ with $ux\in N$ implies $x\in N$ or $u\in
\sqrt{(N:_{R}M)}$ \cite{SYZ1994}. Other variants have also been studied; for
instance, classical primary submodules, introduced in \cite{Baz}, require that
for $u,v\in R$, $x\in M$ with $uvx\in N$ implies $ux\in N$ or $v^{k}x\in N$
for some $k\geq1$. Given a submodule $N$ of an $R$-module $M$, its
$M$-radical, denoted by $M$-$rad(N)$ is defined by the intersection of all
prime submodules of $M$, containing $N.$ Further extensions include
$2$-absorbing and $2$-absorbing primary submodules (see \cite{Dar,Pay,M}),
which adapt the corresponding ideal-theoretic conditions to the module setting
via $(N:_{R}M)$ and $M$-$rad(N)$. More recently, sdf-absorbing submodules were
introduced in \cite{Khashan2025}. By $Ann(x)$, we denote the set of $r\in R$
satisfying $rx=0.$ A proper submodule $N$ of $M$ is sdf-absorbing if whenever
$x\in M$ and $u,v\in R\setminus\mathrm{Ann}(x)$ such that $(u^{2}-v^{2})x\in
N$ implies $(u-v)x\in N$ or $(u+v)x\in N$.

Motivated by these developments, we propose a new class of submodules, called
gsdf-absorbing submodules which is a generalization of both classical primary
and sdf-absorbing submodules. A proper submodule $N$ of $M$ is said to be
\emph{generalized square-difference factor absorbing submodules }(briefly,
gsdf-absorbing submodules) if whenever $u,v\in R$, $x\in M$ with $(u^{2}%
-v^{2})x\in N$ implies $(u-v)x\in N$ or $(u+v)^{k}x\in N$ for some $k\geq1$.
We investigate their properties, provide various characterizations and study
their behavior with respect to several constructions, including localization,
factor modules, intersections, and direct products, as well as idealization
and amalgamation.

As an application, we characterize all gsdf-absorbing submodules of $%
\mathbb{Z}
$-module $%
\mathbb{Z}
$ (see Theorem \ref{Zn}). We prove that this occurs for $N=n%
\mathbb{Z}
$ precisely when $n=p^{k}$ or $n=2p^{k}$, where $p$ is prime and $k\geq1$.
This illustrates the influence of the arithmetic structure of $n$ on the
gsdf-absorbing condition. These results place gsdf-absorbing submodules within
the broader family of absorbing-type generalizations and clarify their
relationship with previously studied notions.

\section{Generalized square-difference Factor Absorbing Submodules}

We now introduce the main concept of this paper, namely \emph{generalized
square-difference factor absorbing submodules}, which constitute the central
focus of our study.

\begin{definition}
Let $R$ be a ring and $M$ an $R$-module. A proper submodule $N$ of $M$ is
called a \emph{generalized square-difference factor absorbing submodule}
(abbreviated as \emph{gsdf-absorbing submodule}) if, for all $u ,v \in R$ and
$x \in M$, the condition $(u ^{2} - v ^{2})x \in N$ implies that either $(u -v
)x \in N$ or $(u +v )^{k} x \in N$ for some positive integer $k$.
\end{definition}

The following examples serve to illustrate the concept of gsdf-absorbing submodules.

\begin{example}
\ 
\end{example}

\begin{enumerate}
\item Let $R$ be a ring of characteristic $2$ (i.e., $1+1=0$ in $R$). Then
every proper submodule of an $R$-module is gsdf-absorbing. Indeed, if $N$ is a
proper submodule of $M$ and $u,v\in R$, $x\in M$ satisfy $(u^{2}-v^{2})x\in
N$, then using $char(R)=2$, we have $(u+v)^{2}x=(u^{2}-v^{2})x\in N$.

\item In a reduced $R$-module $M$ (i.e., for $x\in M$ and $u\in R$, $u^{2}x=0$
implies $ux=0$), the zero submodule is gsdf-absorbing if and only if it is sdf-absorbing.

\item Let $R$ be a von Neumann regular ring. In this case, a proper submodule
$N$ of an $R$-module $M$ is gsdf-absorbing if and only if it is sdf-absorbing.
Indeed, in a von Neumann regular ring every ideal is radical (i.e., $\sqrt
{I}=I$), which implies that $\sqrt{(N:_{R}x)}=(N:_{R}x)$, making the
definitions of gsdf-absorbing and sdf-absorbing submodules coincide.
\end{enumerate}

The diagram below situates gsdf-absorbing submodules within the existing
hierarchy of submodule classes, highlighting their close connections to
related notions.

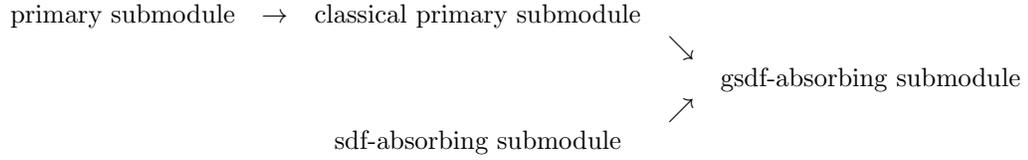
\begin{center}
$%
\begin{array}
[c]{ccccc}%
\text{primary submodule} & \rightarrow & \text{classical primary submodule} &
& \\
&  &  & \searrow & \\
&  &  &  & \text{gsdf-absorbing submodule}\\
&  &  & \nearrow & \\
&  & \text{sdf-absorbing submodule} &  &
\end{array}
\bigskip$
\end{center}

The following examples illustrate that the arrows in the previous diagram are
not reversible.

\begin{example}
\ 
\end{example}

\begin{enumerate}
\item Consider the $\mathbb{Z}$-module $\mathbb{Z}$ and its submodule
$8\mathbb{Z}$. This submodule is gsdf-absorbing. Indeed, let $u,v,x\in
\mathbb{Z}$ satisfy $(u^{2}-v^{2})x\in8\mathbb{Z}$. If $(u-v)x\notin
8\mathbb{Z}$, then $u+v\in2\mathbb{Z}$, so that $(u+v)^{k}x\in8\mathbb{Z}$ for
some $k\geq1$. However, $8\mathbb{Z}$ is \emph{not} sdf-absorbing: for
instance, $(3^{2}-1^{2})\cdot1=8\in8\mathbb{Z}$, but neither $(3-1)\cdot1=2$
nor $(3+1)\cdot1=4$ belong to $8\mathbb{Z}$.

\item Consider $\mathbb{Z}$-module $\mathbb{Z}_{12}.$ Then, the submodule
$(6)$ is gsdf-absorbing in $\mathbb{Z}$-module $\mathbb{Z}_{12}$. Indeed, let
$x\in\mathbb{Z}_{12}$ and $u,v\in\mathbb{Z}$ such that $(u^{2}-v^{2})x\in(6)$
and $(u-v)x\notin(6)$. Then $(u+v)-(u-v)=2v\equiv0\pmod{2}$, so $(u+v)\equiv
(u-v)\pmod{2}$. Hence, we have the following two cases:

\text{Case I.} Suppose both $u+v$ and $u-v$ are even. Since $(u-v)x\notin(6)$,
we have $3\nmid(u-v)$ and $3\nmid x$. But $3\mid(u+v)(u-v)x$, which forces
$3\mid(u+v)$, and therefore $(u+v)x\in(6)$.

\text{Case II.} Suppose both $u+v$ and $u-v$ are odd. Then $(u+v)(u-v)x\in
(6)\subset(2)$ implies $2\mid x$. Since $(u-v)x\notin(6)$, we must have
$3\nmid(u-v)$. But $3\mid(u+v)(u-v)x$ again forces $3\mid(u+v)x$, so
$(u+v)x\in(6)$.

Thus, $(6)$ is gsdf-absorbing in $\mathbb{Z}_{12}$. However, $(6)$ is
\emph{not} classical primary as $2\cdot3\cdot1\in(6)$, but neither $2\cdot
1\in(6)$ nor $3^{n}\cdot1\in(6)$ for any $n\geq1$.
\end{enumerate}

The following result establishes a condition under which a classical primary
submodule becomes a gsdf-absorbing submodule.

\begin{proposition}
Let $R$ be a ring, $2\in U(R)$ and let $M$ be an $R$-module. A proper
submodule $N$ of $M$ is gsdf-absorbing if and only if $N$ is a classical
primary submodule of $M.$
\end{proposition}

\begin{proof}
Suppose that $N$ is a gsdf-absorbing submodule of $M$. Given $u,v\in R$ and
$x\in M$ with $uvx\in N$, set $a=\frac{u+v}{2},$ $b=\frac{v-u}{2}.$ Then
$a^{2}-b^{2}=uv$, so $(a^{2}-b^{2})x=uvx\in N$. By the gsdf-absorbing
property, either $(a-b)x=ux\in N$ or $(a+b)^{k}x=v^{k}x\in N$ for some
$k\geq1$, as needed. Conversely, suppose that $N$ is a classical primary
submodule of $M$. For any $a,b\in R$ and $x\in M$ with $(a^{2}-b^{2})x\in N$,
take $u=a-b$ and $v=a+b$. Then $uvx=(a^{2}-b^{2})x$, and the hypothesis gives
$ux=(a-b)x\in N$ or $v^{k}x=(a+b)^{k}x\in N$, verifying that $N$ is gsdf-absorbing.
\end{proof}

Let $R$ be a ring and $M$ an $R$-module. For a submodule $N\subseteq M$, an
element $r\in R$, and $x\in M$, by $(N:_{R}x)$ and $(N:_{M}r)$, we denote the
ideal $\{r\in R\mid rx\in N\}$ of $R$ and the submodule $\{x\in M\mid rx\in
N\}$ of $M,$ respectively.

The following theorem presents several equivalent characterizations of
gsdf-absorbing submodules in arbitrary modules.

\begin{theorem}
\label{eq} Let $M$ be an $R$-module and $N$ a proper submodule of $M$. Then
the following statements are equivalent:

\begin{enumerate}
\item $N$ is a gsdf-absorbing submodule of $M$.

\item For every $r \in R$ with $rM \nsubseteq N$, the submodule $(N :_{M} r)$
is gsdf-absorbing.

\item For every $x \in M \setminus N$, the ideal $(N :_{R} x)$ is
sdf-absorbing primary.

\item For every finitely generated submodule $K$ of $M$ with $K \nsubseteq N$,
the ideal $(N:_{R} K)$ is sdf-absorbing primary.
\end{enumerate}
\end{theorem}

\begin{proof}
$(1)\Rightarrow(2)$: Let $r\in R$ with $rM\nsubseteq N$. Take $u,v\in R$ and
$x\in M$ such that $(u^{2}-v^{2})x\in(N:_{M}r)\ \text{meaning}\ (u^{2}%
-v^{2})rx\in N.$ Since $N$ is gsdf-absorbing, either $(u-v)rx\in
N\ \text{or}\ (u+v)^{k}rx\in N$ for some $k\geq1$. Hence, $(u-v)x\in(N:_{M}r)$
or $(u+v)^{k}x\in(N:_{M}r),$ showing that $(N:_{M}r)$ is gsdf-absorbing.

$(2)\Rightarrow(3)$: First, note that since $1_{R}M=M\nsubseteq N$,
$(N:_{M}1_{R})=N$ is gsdf-absorbing. Let $x\in M\setminus N$. If $u^{2}%
-v^{2}\in(N:_{R}x)$, the gsdf-absorbing property gives $(u-v)x\in
N\ \text{or}\ (u+v)^{k}x\in N$ for some $k\geq1$, hence $u-v\in(N:_{R}%
x)\ \text{or}\ (u+v)^{k}\in(N:_{R}x),$ so $(N:_{R}x)$ is sdf-absorbing primary.

$(3)\Rightarrow(4)$: Let $K=\langle x_{1},x_{2},\dots,x_{n}\rangle$, where
$x_{1},x_{2},\dots,x_{n}\in M$. Assume that $(u^{2}-v^{2})K\subseteq N$, but
$(u-v)K\nsubseteq N$. Then, there exists at least one generator $x_{i}$ such
that $(u-v)x_{i}\notin N$. Without loss of generality, suppose that
$(u-v)x_{i}\notin N\text{ for }i=1,\dots,s,$ and $(u-v)x_{j}\in N\text{ for
}j=s+1,\dots,n.$ Since $(u^{2}-v^{2})x_{i}\in N$ and $(u-v)x_{i}\notin N$, for
each $i=1,\dots,s$, there exists $t_{i}\geq1$ such that $(u+v)^{t_{i}}x_{i}\in
N.$ Set $t=\max\{t_{1},t_{2},\dots,t_{s}\}.$ If $s=n$, then we obtain
$(u+v)^{t}K\subseteq N,$ and the conclusion follows.

Now, assume that $s<n$. For each element $x_{j}$ in $\{x_{s+1},...,x_{n}\}.$
Then, since $(u-v)x_{1}\notin N$ and $(u-v)x_{j}\in N,$ we have $(u-v)(x_{1}%
+x_{j})\notin N,$ and clearly $(u^{2}-v^{2})(x_{1}+x_{j})\in N.$ Therefore,
there exists $t_{j}\geq1$ for each $j=s+1,\dots,n$ such that $(u+v)^{t_{j}%
}(x_{1}+x_{j})\in N.$ Put $t^{\prime}=\max\{t_{1},t_{2},...,t_{s}%
,t_{s+1},...,t_{n}\}=\max\{t,t_{s+1},...,t_{n}\}.$ Hence, $(u+v)^{t^{\prime}%
}x_{j}\in N$ for all $j=s+1,\dots,n.$ Consequently, we have $(u+v)^{t^{\prime
}}x_{i}\in N$ for all $i=1,\dots,n,$ and thus, $(u+v)^{t}K\subseteq N,$ as required.

$(4)\Rightarrow(1)$: Let $x\in M$ and $(u^{2}-v^{2})x\in N$. If $x\notin N$,
set $K=Rx$; then $K\nsubseteq N$, so $(N:_{R}K)$ is sdf-absorbing primary.
Since $u^{2}-v^{2}\in(N:_{R}K)=(N:_{R}x)$, it follows that $u-v\in
(N:_{R}x)\text{ or }(u+v)^{k}\in(N:_{R}x),$ hence $(u-v)x\in N$ or
$(u+v)^{k}x\in N.$ Therefore $N$ is gsdf-absorbing.
\end{proof}

Next, we present a condition for a gsdf-absorbing submodule to be prime, see \cite{M1992}. We
say that a gsdf-absorbing submodule $N$ of $M$ is maximal if it is not
properly contained in any gsdf-absorbing submodule of $M$.

\begin{proposition}
Any maximal gsdf-absorbing submodule of an $R$-module $M$ is prime.
\end{proposition}

\begin{proof}
Let $u\in R$ and $x\in M$ such that $ux\in N$ and $u\not \in (N:_{R}M).$ Then,
$uM\not \subseteq N$ and from Theorem \ref{eq}, $(N:_{M}u)$ is a
gsdf-absorbing submodule of $M.$ By the maximality of $N$, we conclude that
$x\in(N:_{M}u)=N.$ Thus, $N$ is a prime submodule of $M.$
\end{proof}

Motivated by the primary decomposition, we say that a submodule $N$ of $M$
\emph{admits a gsdf-absorbing decomposition} if $N=\bigcap_{i=1}^{n}Q_{i},$
where each $Q_{i}$ is a gsdf-absorbing submodule of $M$.

\begin{remark}
If $M$ is a Noetherian $R$-module, then every proper submodule $N\subseteq M$
admits a gsdf-absorbing decomposition.
\end{remark}

\begin{proof}
Since $M$ is Noetherian, every submodule $N$ has a primary decomposition
$N=\bigcap_{i=1}^{n}Q_{i},$ where each $Q_{i}$ is a primary submodule of $M$
(see \cite[page 423]{Lang}). As every primary submodule is also
gsdf-absorbing, this decomposition provides a gsdf-absorbing decomposition for
$N$.
\end{proof}

However, gsdf-absorbing decompositions are not necessarily unique. For
example, in the $\mathbb{Z}$-module $\mathbb{Z}_{24}$, the submodule
$12\mathbb{Z}$ admits two distinct gsdf-absorbing decompositions (see
\cite[Example 3]{Khashan2025}): $12\mathbb{Z}=3\mathbb{Z}\cap4\mathbb{Z}\text{
and }12\mathbb{Z}=6\mathbb{Z}\cap4\mathbb{Z}.$

\begin{proposition}
Let $I$ be a principal ideal of a ring $R$, $M$ a $R$-module, and $N$ a proper
submodule of $IM$. Then $N$ is gsdf-absorbing in $IM$ if and only if
$(N:_{M}I)$ is gsdf-absorbing in $M$.
\end{proposition}

\begin{proof}
Assume $N$ is gsdf-absorbing in $IM$, where $I=(i_{0}).$ Take any $x\in M$ and
$u,v\in R$ such that $(u^{2}-v^{2})x\in(N:_{M}I),\ \text{i.e., }(u^{2}%
-v^{2})i_{0}x\in N.$ The gsdf-absorbing property of $N$ implies $(u-v)i_{0}%
x\in N\ \text{or}\ (u+v)^{k}i_{0}x\in N\text{ for some }k\geq1.$ Hence,
$(u-v)x\in(N:i_{0})\ \text{or}\ (u+v)^{k}x\in(N:i_{0}),$ showing that $(N:I)$
is sdf-absorbing primary in $M$.

Conversely, assume $(N:_{M}I)$ is gsdf-absorbing in $M$. Let $x_{0}\in IM$ and
$u,v\in R$ with $(u^{2}-v^{2})x_{0}\in N.$ Since $x_{0}\in IM$ and $I=(i_{0}%
)$, it follows that $x_{0}=i_{0}x^{\prime}$ for some $x^{\prime}\in M.$ Then
$(u^{2}-v^{2})i_{0}x^{\prime}\in N$, whence $(u^{2}-v^{2})x^{\prime}\in
(N:_{M}I).$ By the gsdf-absorbing property of $(N:_{M}I)$, either
$(u-v)x^{\prime}\in(N:_{M}I)\ \text{or}\ (u+v)^{k}x^{\prime}\in(N:_{M}I),$ for
some $k\geq1$. Hence $(u-v)Ix^{\prime}\subseteq N$ or $(u+v)^{k}Ix^{\prime
}\subseteq N,$ for some $k\geq1$ and thus, $N$ is gsdf-absorbing in $IM$ since
$x_{0}\in Ix^{\prime}$.
\end{proof}

We now investigate when the zero submodule $\{0\}$ is a gsdf-absorbing
submodule of the $\mathbb{Z}$-module $\mathbb{Z}_{n}$. To this end, we first
prove the following lemma.

\begin{lemma}
\label{l1} Let $n \in\mathbb{N}$. If $\{0\}$ is a gsdf-absorbing submodule of
$\mathbb{Z}_{n}$, then $n$ has at most one odd prime divisor.
\end{lemma}

\begin{proof}
Assume that $\{0\}$ is gsdf-absorbing in $\mathbb{Z}_{n}$, and suppose, for
contradiction, that $n$ has at least two distinct odd prime divisors. Write
$n=pqt$, where $p$ and $q$ are distinct odd primes and $t\in\mathbb{N}$. Set
$u=p+q$, $v=p-q$, and $x=\overline{t}\in\mathbb{Z}_{n}$. Then $u^{2}%
-v^{2}=(p+q)^{2}-(p-q)^{2}=4pq,$ so that $(u^{2}-v^{2})x=4pq\,\overline
{t}=\overline{4pqt}=\overline{0}\text{ in }\mathbb{Z}_{n}.$ However,
$(u-v)x=2q\,\overline{t}=\overline{2qt}\neq\overline{0}\text{ and }%
(u+v)^{k}x=(2p)^{k}\,\overline{t}\neq\overline{0}\text{ for all }k\geq1.$ This
contradicts to the gsdf-absorbing property of $\{0\}$. Therefore, $n$ has at
most one odd prime divisor.
\end{proof}

\begin{lemma}
\label{l2} Let $n \in\mathbb{N}$. If $n = 2^{k} p^{s}$ with $k \ge2$, $s \ge
1$, and $p$ an odd prime, then the zero submodule $\{0\}$ is \emph{not} a
gsdf-absorbing submodule of the $\mathbb{Z}$-module $\mathbb{Z}_{n}$.
\end{lemma}

\begin{proof}
Set $u=2^{k-2}+p^{s},\ v=2^{k-2}-p^{s},\ x=\overline{1}\in\mathbb{Z}_{n}.$
Then $(u^{2}-v^{2})\cdot\overline{1}=2^{k}\cdot p^{s}\,\overline{1}%
=\overline{0}\quad\text{in }\mathbb{Z}_{n}.$ However, $(u-v)\cdot\overline
{1}=2p^{s}\,\overline{1}\neq\overline{0},\quad\text{and}\quad(u+v)^{t}%
\cdot\overline{1}=2^{(k-1)t}\,\overline{1}\neq\overline{0}\text{ for all
}t\geq1.$ Hence, the zero submodule $\{0\}$ fails the gsdf-absorbing condition
in $\mathbb{Z}_{n}$.
\end{proof}

We are now ready to establish the main characterization.

\begin{theorem}
\label{t3} The zero submodule is a gsdf-absorbing submodule of the
$\mathbb{Z}$-module $\mathbb{Z}_{n}$ if and only if $n=p^{k}\ \text{with $p$
prime, or}\ n=2p^{k}\ \text{with $p$ an odd prime}.$
\end{theorem}

\begin{proof}
First, consider $n=p^{k}$ with $p$ prime. Let $u,v\in\mathbb{Z}$ and
$x=\overline{t}\in\mathbb{Z}_{p^{k}}$ such that $(u^{2}-v^{2})\overline
{t}=\overline{0}.$ Then $p^{k}\mid(u-v)(u+v)t$. Since $(p^{k})$ is a primary
ideal of $\mathbb{Z}$, we have either $p^{k}\mid(u-v)t$ or $p\mid(u+v)$. If
$p^{k}\mid(u-v)t$, then $(u-v)x=\overline{0}$ in $\mathbb{Z}_{p^{k}}$. If
$p\mid(u+v)$, write $u+v=pc$ for some $c\in\mathbb{Z}$. Then $(u+v)^{k}%
x=p^{k}c^{k}x=\overline{0}$ in $\mathbb{Z}_{p^{k}}$. Hence, $(u-v)x=\overline
{0}$ or $(u+v)^{k}x=\overline{0}$, proving that $\{0\}$ is gsdf-absorbing in
$\mathbb{Z}_{p^{k}}$.

Next, let $n = 2 p^{k}$ with $p$ an odd prime and $k \ge1$. We will show that
$\{0\}$ is gsdf-absorbing in $\mathbb{Z}_{2 p^{k}}$. Let $u, v \in\mathbb{Z}$
and $x = \overline{t} \in\mathbb{Z}_{2 p^{k}}$ satisfy $(u ^{2} - v ^{2})
\overline{t} = \overline{0} \in\mathbb{Z}_{2 p^{k}}. $

By the Chinese Remainder Theorem, we have $\mathbb{Z}_{n}\cong\mathbb{Z}%
_{2}\times\mathbb{Z}_{p^{k}},$ so every element $x\in\mathbb{Z}_{n}$
corresponds to a pair $(x_{2},x_{p})$ with $x_{2}\in\mathbb{Z}_{2}$ and
$x_{p}\in\mathbb{Z}_{p^{k}}$. Suppose $(u^{2}-v^{2})x=\overline{0}%
\quad\text{in }\mathbb{Z}_{n}.$ Under the decomposition, this is equivalent
to
\[
(u_{2}^{2}-v_{2}^{2})x_{2}=\overline{0}\quad\text{in }\mathbb{Z}_{2}%
,\quad\text{and}\quad(u_{p}^{2}-v_{p}^{2})x_{p}=\overline{0}\quad\text{in
}\mathbb{Z}_{p^{k}}.
\]

Hence, by the previous arguments, we have
\[
(u_{2}-v_{2})x_{2}=\overline{0}\ \text{or}\ (u_{2}+v_{2})^{n}x_{2}%
=\overline{0}\quad\text{in }\mathbb{Z}_{2},
\]
and
\[
(u_{p}-v_{p})x_{p}=\overline{0}\ \text{or}\ (u_{p}+v_{p})^{r}x_{p}%
=\overline{0}\quad\text{in }\mathbb{Z}_{p^{k}},
\]
for some $n,r\geq1$. If $u_{2}$ and $v_{2}$ have the same parity, then both
$(u_{2}-v_{2})x_{2}=\overline{0}$ and $(u_{2}+v_{2})^{n}x_{2}=\overline{0}$
hold. If $u_{2}$ and $v_{2}$ have different parity, then $2\mid x_{2}$, so
again $(u_{2}-v_{2})x_{2}=\overline{0}$ and $(u_{2}+v_{2})^{n}x_{2}%
=\overline{0}$ for some $n\geq2$. Therefore, $\{0\}$ is a gsdf-absorbing
submodule of $\mathbb{Z}_{n}$ when $n=2p^{k}$ with $p$ an odd prime.

The converse  part follows directly from Lemmas \ref{l1} and \ref{l2}.
\end{proof}

\section{Stability of Gsdf-absorbing Submodules}

In this section, we study the stability of gsdf-absorbing submodules under
localization, module homomorphisms, and direct products. Let $M$ be an
$R$-module and $S$ a multiplicatively closed subset of $R$. A proper submodule
$N$ of $M$ is called \emph{$S$-saturated} if for any $s\in S$ and $x\in M$, $s
x \in N $ imply $x \in N. $

\begin{proposition}
\label{S} Let $S$ be a multiplicatively closed subset of $R$ and let $N$ be a
proper submodule of the $R$-module $M$. Then the following statements hold:

\begin{enumerate}
\item If $N$ is a gsdf-absorbing submodule of $M$ and $S^{-1}N \neq S^{-1}M$,
then $S^{-1}N$ is a gsdf-absorbing submodule of $S^{-1}R$-module $S^{-1}M$.

\item If $S^{-1}N$ is a gsdf-absorbing submodule of $S^{-1}M$ and $N$ is
$S$-saturated, then $N$ is a gsdf-absorbing submodule of $M$.
\end{enumerate}
\end{proposition}

\begin{proof}
(1) Let $\frac{x}{s}\in S^{-1}M$ and $\frac{u}{t},\frac{v}{d}\in S^{-1}R$ be
such that $\left(  \left(  \frac{u}{t}\right)  ^{2}-\left(  \frac{v}%
{d}\right)  ^{2}\right)  \frac{x}{s}\in S^{-1}N.$ Then there exists
$s^{\prime}\in S$ such that $s^{\prime}(u^{2}d^{2}-v^{2}t^{2})x=[(ud)^{2}%
-(vt)^{2}](s^{\prime}x)\in N.$ Since $N$ is gsdf-absorbing, we have
$(ud-vt)(s^{\prime}x)\in N\text{ or }(ud+vt)^{k}(s^{\prime}x)\in N$ for some
$k\geq1$. Hence, $\left(  \frac{u}{t}-\frac{v}{d}\right)  \frac{x}{s}%
=\frac{s^{\prime}(ud-vt)x}{s^{\prime}tds}S^{-1}N$ or $\left(  \frac{u}%
{t}+\frac{v}{d}\right)  ^{k}\frac{x}{s}=\frac{s^{\prime}(ud+vt)^{k}%
x}{s^{\prime}(td)^{k}s}S^{-1}N$, and $S^{-1}N$ is gsdf-absorbing in $S^{-1}M$.

(2) Let $x\in M$ and $u,v\in R$ be such that $(u^{2}-v^{2})x\in N$. Then
$\frac{u^{2}-v^{2}}{1}\frac{x}{1}\in S^{-1}N.$ Since $S^{-1}N$ is
gsdf-absorbing in $S^{-1}M$, either $\frac{u-v}{1}\frac{x}{1}\in
S^{-1}N\ \text{or}\ \left(  \frac{u+v}{1}\right)  ^{k}\frac{x}{1}\in S^{-1}N$
for some $k\in\mathbb{Z}^{+}$. If $\frac{u-v}{1}\frac{x}{1}\in S^{-1}N$, then
there exists $s\in S$ such that $s(u-v)x\in N$. Because $N$ is $S$-saturated,
it follows that $(u-v)x\in N$. Similarly, if $\left(  \frac{u+v}{1}\right)
^{k}\frac{x}{1}\in S^{-1}N$, then for some $s^{\ast}\in S$ we have $s^{\ast
}(u+v)^{k}x\in N$, for some $k\geq1.$ Using $S$-saturation again, we conclude
$(u+v)^{k}x\in N$. Therefore, $N$ is a gsdf-absorbing submodule of $M$.
\end{proof}

\begin{proposition}
\label{f} Let $f: M \to M^{\prime}$ be an $R$-module epimorphism, and let $N
\subsetneq M$, $N^{\prime}\subsetneq M^{\prime}$ be proper submodules. Then:

\begin{enumerate}
\item If $N$ is a gsdf-absorbing submodule of $M$ with $\ker(f)\subseteq N$,
then $f(N)$ is a gsdf-absorbing submodule of $M^{\prime}$.

\item If $N^{\prime}$ is a gsdf-absorbing submodule of $M^{\prime}$ and
$f^{-1}(N^{\prime}) \neq M$, then $f^{-1}(N^{\prime})$ is a gsdf-absorbing
submodule of $M$.
\end{enumerate}
\end{proposition}

\begin{proof}
(1) Let $u,v\in R$ and $x^{\prime}\in M^{\prime}$ such that $(u^{2}%
-v^{2})x^{\prime}\in f(N)$. Since $f$ is surjective, there exists $x\in M$
with $f(x)=x^{\prime}$. Then $f((u^{2}-v^{2})x)=(u^{2}-v^{2})x^{\prime}\in
f(N)$ and so $(u^{2}-v^{2})x\in N+\ker(f)\subseteq N.$ As $N$ is
gsdf-absorbing, $(u-v)x\in N$ or $(u+v)^{k}x\in N$ for some $k\geq1$. Hence,
we get either $(u-v)x^{\prime}=f((u-v)x)\in f(N)\text{ or }(u+v)^{k}x^{\prime
k}x)\in f(N),$ proving that $f(N)$ is gsdf-absorbing in $M^{\prime}$.

(2) Let $u,v\in R$ and $m\in M$ such that $(u^{2}-v^{2})x\in f^{-1}(N^{\prime
})$, i.e., $f((u^{2}-v^{2})x)=(u^{2}-v^{2})f(x)\in N^{\prime}$. Since
$N^{\prime}$ is gsdf-absorbing either $(u-v)f(x)=f((u-v)x)\in N^{\prime}$ or
$(u+v)^{k}f(x)=f((u+v)^{k}x)\in N^{\prime}$ for some $k\geq1$. Hence,
$(u-v)x\in f^{-1}(N^{\prime})\text{ or }(u+v)^{k}x\in f^{-1}(N^{\prime}).$
Therefore, $f^{-1}(N^{\prime})$ is gsdf-absorbing in $M$.
\end{proof}

As a consequence of the previous theorem, we present the following results on
division modules.

\begin{corollary}
\label{C43} Let $R$ be a ring and $M_{1},$ $M_{2}$ be $R$-modules.

\begin{enumerate}
\item If $M_{1} \subseteq M_{2}$ and $N$ is a gsdf-absorbing submodule of
$M_{2}$, with $N \neq M_{1}$, then $N \cap M_{1}$ is a gsdf-absorbing
submodule of $M_{1}$.

\item Let $K\subseteq N$ be proper submodules of $M_{1}$. Then $N$ is a
gsdf-absorbing submodule of $M_{1}$ if and only if $N/K$ is a gsdf-absorbing
submodule of $M_{1}/K$.
\end{enumerate}
\end{corollary}

We are now ready to completely determine all gsdf-absorbing submodules of
$\mathbb{Z}$-module $\mathbb{Z}$.

\begin{theorem}
\label{Zn}A proper submodule $N=n\mathbb{Z}$ of $\mathbb{Z}$-module
$\mathbb{Z}$ is a gsdf-absorbing if and only if $n=p^{k}\ \text{with $p$
prime, or}\ n=2p^{k}\ \text{with $p$ an odd prime.}$
\end{theorem}

\begin{proof}
Suppose that $N=n\mathbb{Z}$ is a gsdf-absorbing submodule of $\mathbb{Z}%
$-module $\mathbb{Z}$. Put $K=N=n\mathbb{Z}$ in Corollary \ref{C43}. Then,
$\{0\}=N/N$ is a gsdf-absorbing submodule of $\mathbb{Z}/n\mathbb{Z}\cong \mathbb{Z}_{n}.$
From Theorem \ref{t3}, we have $n=p^{k}\ $with $p$ prime, or$\ n=2p^{k}\ $with
$p$ an odd prime. The converse also follows from Theorem \ref{t3} and
Corollary \ref{C43}$.$
\end{proof}

\begin{proposition}
\label{P44} Let $M$ be an $R$-module.

\begin{enumerate}
\item Let $N_{1}$ and $N_{2}$ be gsdf-absorbing submodules of $M$ such that
$\sqrt{(N_{1}:_{R}x)}=\sqrt{(N_{2}:_{R}x)}$ for all $x\in M\setminus(N_{1}\cap
N_{2}).$ Then $N_{1}\cap N_{2}$ is a gsdf-absorbing submodule of $M$.

\item Let $\{N_{i}\}_{i\in I}$ be a directed family of gsdf-absorbing
submodules of $M$. Then the union $N=\bigcup_{i\in I}N_{i}$ is a
gsdf-absorbing submodule of $M$.
\end{enumerate}
\end{proposition}

\begin{proof}
(1) If $x\in N_{1}\cap N_{2}$, there is nothing to prove. So, assume $x\notin
N_{1}\cap N_{2}$ and let $u,v\in R$ be such that $(u^{2}-v^{2})x\in N_{1}\cap
N_{2}.$ Since $N_{1}$ and $N_{2}$ are gsdf-absorbing, for each $i=1,2$ we have
$(u-v)x\in N_{i}\text{ or }u+v\in\sqrt{(N_{i}:_{R}x)}.$ If $u+v\in\sqrt
{(N_{1}:_{R}x)}=\sqrt{(N_{2}:_{R}x)}$, then $u+v\in\sqrt{(N_{1}\cap N_{2}%
:_{R}x)}.$ Otherwise, if $u+v\notin\sqrt{(N_{i}:_{R}x)}$, then $(u-v)x\in
N_{1}\cap N_{2}$. In either case, the gsdf-absorbing condition holds for
$N_{1}\cap N_{2}$.

(2) This follows immediately from the definition of a directed union and the
fact that each $N_{i}$ is gsdf-absorbing.
\end{proof}

It is worth noting that both the intersection and the sum of two
gsdf-absorbing submodules do not necessarily remain gsdf-absorbing. A simple
counterexample arises in the $\mathbb{Z}$-module $\mathbb{Z}$ with submodules
$N_{1}=3\mathbb{Z}$ and $N_{2}=7\mathbb{Z}$. Observe that $(5^{2}-2^{2}%
)\cdot2\in N_{1}\cap N_{2},$ but $(5-2)\cdot2\notin N_{1}\cap N_{2},$
$(5+2)^{k}\cdot2\notin N_{1}\cap N_{2}\ \text{for any }k\geq1,$ which shows
that $N_{1}\cap N_{2}$ fails to be gsdf-absorbing. Additionally, the sum
$N_{1}+N_{2}=\mathbb{Z}$ is not a proper submodule, and hence it is not
gsdf-absorbing either.

\begin{proposition}
\label{cart_rewrite} Let $N_{1}$ and $N_{2}$ be proper submodules of
$R$-modules $M_{1}$ and $M_{2}$, respectively.
\end{proposition}

\begin{enumerate}
\item If $N_{1}\times N_{2}$ is a gsdf-absorbing submodule of $M_{1}\times
M_{2}$ then $N_{1}$ and $N_{2}$ are gsdf-absorbing submodules of $M_{1}$ and
$M_{2}$, respectively.

\item $N_{1}\times M_{2}$ is gsdf-absorbing in $M_{1}\times M_{2}$ if and only
if $N_{1}$ is gsdf-absorbing in $M_{1}$.

\item $M_{1}\times N_{2}$ is gsdf-absorbing in $M_{1}\times M_{2}$ if and only
if $N_{2}$ is gsdf-absorbing in $M_{2}$.
\end{enumerate}

\begin{proof}
(1) Assume $u,v\in R$ and $x_{1}\in M_{1}$ satisfy $(u^{2}-v^{2})x_{1}\in
N_{1}$. Then $(u^{2}-v^{2})(x_{1},0)\in N_{1}\times N_{2}$. Since $N_{1}\times
N_{2}$ is gsdf-absorbing, we must have either $(u-v)(x_{1},0)\in N_{1}\times
N_{2}$ or $(u+v)^{k}(x_{1},0)\in N_{1}\times N_{2}$ for some $k\geq1$. This
directly implies $(u-v)x_{1}\in N_{1}$ or $(u+v)^{k}x_{1}\in N_{1}$,
establishing that $N_{1}$ is gsdf-absorbing. Similary, $N_{2}$ is
gsdf-absorbing in $M_{2}.$

(2) The sufficiency part similar to (1). For the necessary part, let
$(x_{1},x_{2})\in M_{1}\times M_{2}$ and $u,v\in R$ be such that $(u^{2}%
-v^{2})(x_{1},x_{2})\in N_{1}\times M_{2}$. Then $(u^{2}-v^{2})x_{1}\in N_{1}%
$, and by the gsdf-absorbing property of $N_{1}$, we have either
$(u-v)x_{1}\in N_{1}$ or $(u+v)^{k}x_{1}\in N_{1}$ for some $k\geq1$. Hence,
$(u-v)(x_{1},x_{2})\in N_{1}\times M_{2}$ or $(u+v)^{k}(x_{1},x_{2})\in
N_{1}\times M_{2}$, proving that $N_{1}\times M_{2}$ is gsdf-absorbing in
$M_{1}\times M_{2}$.

(3) The argument for $M_{1}\times N_{2}$ is analogous to part (2).
\end{proof}

If $N_{1}$ and $N_{2}$ are gsdf-absorbing submodules of $M_{1}$ and $M_{2}$
respectively, then $N_{1}\times N_{2}$ need not be gsdf-absorbing in
$M_{1}\times M_{2}.$ Consider $N_{1}=10\mathbb{Z},\ N_{2}=9\mathbb{Z}%
,\ R=\mathbb{Z}=M_{1}=M_{2}.$ Then $N_{1}$ and $N_{2}$ are gsdf-absorbing in
$\mathbb{Z}$-module $\mathbb{Z}$ by Theorem \ref{Zn}. However, $N_{1}\times
N_{2}$ is not gsdf-absorbing as $(4^{2}-1^{2})(2,3)=(30,45)\in N_{1}\times
N_{2}$ but $(4-1)(2,3)=(6,9)\not \in N_{1}\times N_{2},$ and $(4+1)^{k}%
(2,3)=5^{k}(2,3)\not \in N_{1}\times N_{2},$ for all $k\geq1.$ Furthermore,
see the next example.

\begin{example}
Let $R$ be a ring with characteristic $\mathrm{char}(R)=2n-1$ for some integer
$n\geq2$, and let $M_{1}$ and $M_{2}$ be $R$-modules with proper submodules
$N_{1}\subset M_{1}$ and $N_{2}\subset M_{2}$. Even if both $N_{1}$ and
$N_{2}$ are gsdf-absorbing, the product submodule $N_{1}\times N_{2}$ of
$M_{1}\times M_{2}$ need not be gsdf-absorbing.
\end{example}

\begin{proof}
Choose elements in $R\times R$ by $u=(n\cdot1_{R},n\cdot1_{R}),\ v=(-n\cdot
1_{R},n\cdot1_{R}),$ and pick $x_{1}\in M_{1}\setminus N_{1}$ and $x_{2}\in
M_{2}\setminus N_{2}$. Observe that $u^{2}-v^{2}=(n^{2}\cdot1_{R},n^{2}%
\cdot1_{R})-(n^{2}\cdot1_{R},n^{2}\cdot1_{R})=(0_{R},0_{R}),$ so $(u^{2}%
-v^{2})(x_{1},x_{2})=(0,0)\in N_{1}\times N_{2}$. However, we have
$u-v=(2n\cdot1_{R},0_{R})=(1_{R},0_{R}),$ $u+v=(0_{R},2n\cdot1_{R}%
)=(0_{R},1_{R})$. Then $(u-v)(x_{1},x_{2})=(x_{1},0)\notin N_{1}\times N_{2},$
$(u+v)^{k}(x_{1},x_{2})=(0,x_{2})\notin N_{1}\times N_{2}$ for any $k\geq1$.
This shows that $N_{1}\times N_{2}$ fails the gsdf-absorbing condition, even
though both factors are gsdf-absorbing submodules.
\end{proof}

\section{Idealizations and Amalgamations}

In this section, we review two classical constructions in commutative algebra:
the idealization of a module and the amalgamation of rings along an ideal.
These constructions are often used to produce examples, counterexamples, and
to study how algebraic properties transfer between structures.

Let $R$ be a commutative ring and $M$ an $R$-module. The \emph{idealization}
(or \emph{trivial extension}) of $R$ by $M$ is the ring $R \ltimes M = \{ (r
,x) \mid r \in R, x \in M \}, $ with addition defined componentwise and
multiplication given by $(u,x)(v,y) = (u v, u y + v x) $ (see \cite{AW}). For any submodule $N
\subseteq M$ and any ideal $I$ of $R$, one has $\sqrt{I \ltimes N} = \sqrt{I}
\ltimes M. $

\begin{proposition}
\label{Prop} Let $I$ be a proper ideal of $R$ and $N$ a submodule of an
$R$-module $M$. Then the following statements hold:

\begin{enumerate}
\item If $I \ltimes N$ is an sdf-absorbing primary ideal of $R \ltimes M$,
then $I$ is an sdf-absorbing primary ideal of $R$.

\item $I$ is an sdf-absorbing primary ideal of $R$ if and only if $I \ltimes
M$ is an sdf-absorbing primary ideal of $R \ltimes M$.
\end{enumerate}
\end{proposition}

\begin{proof}
(1) Assume that $I \ltimes N$ is an sdf-absorbing primary ideal of $R \ltimes
M$, and let $u ,v \in R$ satisfy $u^{2} - v ^{2} \in I$. Then $(u,0)^{2} -
(v,0)^{2} = (u ^{2} - v ^{2}, 0) \in I \ltimes N. $ By the sdf-absorbing
primary property, either $(u -v ,0) \in I \ltimes N$ or $(u +v ,0)^{k} \in I
\ltimes N$ for some $k \ge1$. This immediately gives $u-v \in I$ or $(u+v
)^{k} \in I$, as required.

\medskip(2) Suppose that $I$ is an sdf-absorbing primary ideal of $R$, and let
$(u,x), (v ,y) \in R \ltimes M$ satisfy $(u ,x)^{2} - (v ,y)^{2} \in I \ltimes
M. $ Then $(u^{2} -v ^{2}, 2ux - 2v y) \in I \ltimes M. $ Since $I$ is
sdf-absorbing primary in $R$, either $u -v \in I$ or $(u +v )^{k} \in I$ for
some $k \ge1$. Hence, either $(u ,x) - (v ,y) = (u -v , x-y) \in I \ltimes M
\ \text{or} \ (u +v ,x+y)^{k} \in I \ltimes M, $ where $(u +v ,x+y)^{k} = ((u
+v )^{k}, k(u +v )^{k-1}(x+y))$. This shows that $I \ltimes M$ is
sdf-absorbing primary in $R \ltimes M$. The converse follows directly from
part (1).
\end{proof}

We note that the converse of Proposition \ref{Prop} (1) does not hold in
general: even if $I$ is an sdf-primary ideal of $R$, the ideal $I\ltimes N$ in
$R\ltimes M$ may fail to be sdf-absorbing primary, as illustrated in the next example.

\begin{example}
Consider the ideal $I=3\mathbb{Z}$ of $\mathbb{Z}$, the submodule
$N=2\mathbb{Z}$ of the $\mathbb{Z}$-module $\mathbb{Z}$, and the idealization
$I\ltimes N\subset\mathbb{Z}\ltimes\mathbb{Z}$. Take $(u,x_{1})=(2,1),\quad
(v,x_{2})=(-1,0)\in\mathbb{Z}\ltimes\mathbb{Z}.$ Then $(u,x_{1})^{2}%
-(v,x_{2})^{2}=(u^{2}-v^{2},$ $2ux_{1}-2vx_{2})=(3,4)\in I\ltimes N,$ but
$(u,x_{1})-(v,x_{2})=(3,1)\notin I\ltimes N,$ $((u,x_{1})+(v,x_{2}%
))^{k}=(1,\ast)\notin I\ltimes N$ for any $k\geq1$. Therefore, $I\ltimes N$
fails to be an sdf-absorbing primary ideal in $R\ltimes M$.
\end{example}

It is important to note that even if $I\ltimes N$ is sdf-absorbing primary in
$R\ltimes M$, the submodule $N$ itself need not satisfy the gsdf-absorbing condition.

\begin{example}
Consider the ideal $I=2\mathbb{Z}$ of $\mathbb{Z}$, and the submodule
$N=n\mathbb{Z}$ of $\mathbb{Z}$-module $\mathbb{Z}$ with $n\geq2$. In the
idealization $I\ltimes N$, any $(u,x_{1}),(v,x_{2})\in\mathbb{Z}%
\ltimes\mathbb{Z}$ satisfying
\[
(u,x_{1})^{2}-(v,x_{2})^{2}=(u^{2}-v^{2},2ux_{1}-2vx_{2})\in I\ltimes N
\]
must have $u^{2}-v^{2}\in I=2\mathbb{Z}$, implying $u$ and $v$ are of the same
parity. Consequently, $2\mid(u-v)$ and $2\mid(u+v)$, which ensures $(u-v)\in
I$ and $(u+v)^{k}\in I$ for all $k\geq1$. Therefore, $(u+v,x_{1}+x_{2}%
)^{k}=(u+v)^{k},k(u+v)^{k-1}(x_{1}+x_{2}))\in I\ltimes N.$ However, $N$ itself
may fail to be gsdf-absorbing. For example, if $N=42\mathbb{Z}$, then
$(5^{2}-2^{2})\cdot2=42\notin N$ and $(5+2)^{k}\cdot2=7^{k}\cdot2\notin N$ for
all $k\geq1$, showing that $N$ is not gsdf-absorbing in $\mathbb{Z}$.
\end{example}

We recall the notion of an \emph{amalgamation of rings}, see \cite{ElKhalfi-Kim-Mahdou}. Let $f:R_{1}%
\rightarrow R_{2}$ be a ring homomorphism between commutative rings with
identity, and let $J$ be an ideal of $R_{2}$. The \emph{amalgamation of
$R_{1}$ with $R_{2}$ along $J$ with respect to $f$} is the subring of
$R_{1}\times R_{2}$ defined by $R_{1}\bowtie^{f}J=\{(u,f(u)+j)\mid u\in
R_{1},\,j\in J\}.$

Let $M_{1}$ be an $R_{1}$-module and $M_{2}$ be an $R_{2}$-module. The
\emph{amalgamation of $M_{1}$ and $M_{2}$ along $J$ with respect to a module
homomorphism $\varphi:M_{1}\rightarrow M_{2}$} is defined as in \cite{Rachida} by%

\[
M_{1}\Join^{\varphi}JM_{2}=\bigl\{(x_{1},\varphi(x_{1})+x_{2})\mid x_{1}\in
M_{1},\,x_{2}\in JM_{2}\bigr\}.
\]

This construction naturally equips $M_{1}\Join^{\varphi}JM_{2}$ with an
$(R_{1}\Join^{f}J)$-module structure, where scalar multiplication is given by
\[
(u,f(u)+j)\cdot(x_{1},\varphi(x_{1})+x_{2})=\bigl(ux_{1},\varphi
(ux_{1})+f(u)x_{2}+j\varphi(x_{1})+jx_{2}\bigr).
\]

Let $N_{1} \subseteq M_{1}$ and $N_{2} \subseteq M_{2}$ be submodules. We
define the following subsets of $M_{1} \Join^{\varphi}J M_{2}$:
\[
N_{1} \Join^{\varphi}J M_{2} = \bigl\{ (x_{1}, \varphi(x_{1}) + x_{2}) \in
M_{1} \Join^{\varphi}J M_{2} \mid x_{1} \in N_{1}, x_{2} \in JM_{2} \bigr\},
\]
\[
\overline{N_{2}}^{\varphi}= \bigl\{ (x_{1}, \varphi(x_{1}) + x_{2}) \in M_{1}
\Join^{\varphi}J M_{2} \mid\varphi(x_{1}) + x_{2} \in N_{2} \bigr\}.
\]
Both sets are submodules of $M_{1} \Join^{\varphi}J M_{2}$, and we will use
this notation throughout the section.

In the following theorems, we provide necessary and sufficient conditions
under which $N_{1} \Join^{\varphi}J M_{2}$ and $\overline{N_{2}}^{\varphi}$
are sdf-absorbing submodules of $M_{1} \Join^{\varphi}J M_{2}$, in the spirit
of the approach used in \cite[Theorem 9]{Khashan2025}.

\begin{theorem}
Let $f:R_{1}\rightarrow R_{2}$ be a ring homomorphism, $J$ an ideal of $R_{2}%
$, $M_{1}$ an $R_{1}$-module, $M_{2}$ an $R_{2}$-module (considered as an
$R_{1}$-module via $f$), and $\varphi:M_{1}\rightarrow M_{2}$ an $R_{1}%
$-module homomorphism. Define the $(R_{1}\bowtie^{f}J)$-module
\[
M_{1}\Join^{\varphi}JM_{2}=\bigl\{(x_{1},\varphi(x_{1})+x_{2})\mid x_{1}\in
M_{1},\,x_{2}\in JM_{2}\bigr\}.
\]
Let $N_{1}$ be a proper submodule of $M_{1}$. Then, the following statements
are equivalent:

\begin{enumerate}
\item $N_{1} \Join^{\varphi}J M_{2}$ is a gsdf-absorbing submodule of $M_{1}
\Join^{\varphi}J M_{2}$.

\item $N_{1}$ is a gsdf-absorbing submodule of $M_{1}$.
\end{enumerate}
\end{theorem}

\begin{proof}
Let $N_{1}$ be a proper submodule of $M_{1}$. Notice first that $N_{1}$ is
proper in $M_{1}$ if and only if $N_{1} \Join^{\varphi}J M_{2}$ is proper in
$M_{1} \Join^{\varphi}J M_{2}$.

\text{(1) $\Rightarrow$ (2):} Let $N_{1}$ be a proper submodule of $M_{1}$.
Notice first that $N_{1}$ is proper in $M_{1}$ if and only if $N_{1}%
\Join^{\varphi}JM_{2}$ is proper in $M_{1}\Join^{\varphi}JM_{2}$. Assume that
$N_{1}\Join^{\varphi}JM_{2}$ is a gsdf-absorbing submodule of $M_{1}%
\Join^{\varphi}JM_{2}$. Take arbitrary $x_{1}\in M_{1}$ and $u_{1},v_{1}\in
R_{1}$ such that $(u_{1}^{2}-v_{1}^{2})x_{1}\in N_{1}.$ Consider the element
$(x_{1},\varphi(x_{1}))\in M_{1}\Join^{\varphi}JM_{2}$. Then
\[
\big((u_{1},f(u_{1}))^{2}-(v_{1},f(v_{1}))^{2}\big)\cdot(x_{1},\varphi
(x_{1}))=((u_{1}^{2}-v_{1}^{2})x_{1},\varphi((u_{1}^{2}-v_{1}^{2})x_{1}))\in
N_{1}\Join^{\varphi}JM_{2}.
\]
Since $N_{1}\Join^{\varphi}JM_{2}$ is gsdf-absorbing, we conclude that
either\newline$((u_{1}-f(u_{1}))-(v_{1}-f(v_{1})))(x_{1},\varphi(x_{1}))\in
N_{1}\Join^{\varphi}JM_{2}$ or\newline$((u_{1}-f(u_{1}))+(v_{1}-f(v_{1}%
)))^{k}(x_{1},\varphi(x_{1}))\in N_{1}\Join^{\varphi}JM_{2}$, for some
$k\geq1.$ Hence $(u_{1}-v_{1})x_{1}\in N_{1}\ \text{or}\ (u_{1}+v_{1}%
)^{k}x_{1}\in N_{1}\text{ for some }k\geq1.$ Thus, $N_{1}$ is gsdf-absorbing
in $M_{1}$.

\text{(2) $\Rightarrow$ (1):} Now suppose that $N_{1}$ is a gsdf-absorbing
submodule of $M_{1}$.

Let $(u_{1},f(u_{1})+j),(v_{1},f(v_{1})+j^{\prime})\in R_{1}\Join^{f}J$ and
$(x_{1},\varphi(x_{1})+x_{2})\in M_{1}\Join^{\varphi}JM_{2}$ satisfy
\[
\big((u_{1},f(u_{1})+j)^{2}-(v_{1},f(v_{1})+j^{\prime2}\big)\cdot
(x_{1},\varphi(x_{1})+x_{2})\in N_{1}\Join^{\varphi}JM_{2}.
\]
Looking at the first component gives $(u_{1}^{2}-v_{1}^{2})x_{1}\in N_{1}$.
Because $N_{1}$ is gsdf-absorbing, we have either $(u_{1}-v_{1})x_{1}\in
N_{1}$ or $(u_{1}+v_{1})^{k}x_{1}\in N_{1}$ for some $k\geq1$. Consequently,
since $x_{2}\in JM_{2}$, one can easily check that the corresponding
operations in $M_{1}\Join^{\varphi}JM_{2}$ satisfy
\[
\big((u_{1},f(u_{1})+j)-(v_{1},f(v_{1})+j^{\prime})\big)\cdot(x_{1}%
,\varphi(x_{1})+x_{2})\in N_{1}\Join^{\varphi}JM_{2}%
\]
or
\[
\big((u_{1},f(u_{1})+j)+(v_{1},f(v_{1})+j^{\prime})\big)^{k}\cdot
(x_{1},\varphi(x_{1})+x_{2})\in N_{1}\Join^{\varphi}JM_{2}.
\]
Thus, $N_{1}\Join^{\varphi}JM_{2}$ is a gsdf-absorbing submodule of
$M_{1}\Join^{\varphi}JM_{2}$.
\end{proof}
\section*{Conflict of Interest}
The authors declare that they have no conflict of interest.


\begin{thebibliography}{99}                                                                                               %


\bibitem {AB2024}D. D. Anderson, A. Badawi, and J. Coykendall,
\textit{Square-difference factor absorbing ideals of a commutative ring},
Communications in Algebra, 52(4) (2024), 1542--1550.

\bibitem {AW}D. D. Anderson and M. Winders, Idealization of a module,
\textit{Journal of Commutative Algebra}, 1 (2009), 69--83.

\bibitem {AM1969}M. F. Atiyah and I. G. Macdonald, \textit{Introduction to
Commutative Algebra}, Addison-Wesley Publishing Co., Reading, Mass.-London-Don
Mills, Ont., 1969.


\bibitem {B2007}A. Badawi, \textit{On 2-absorbing ideals of commutative
rings}, Bull. Austral. Math. Soc., 75 (2007), 417--429.


\bibitem {BTY}A. Badawi, U. Tekir, and E. Yetkin, On 2-absorbing primary
ideals in commutative rings, \textit{Bulletin of the Korean Mathematical
Society}, 51(4) (2014), 1163--1173.

\bibitem {Baz}M. Baziar, M. Behboodi, Classical primary submodules and
decomposition theory of modules, Journal of Algebra and its Applications, 8(3)
(2009), 351-362.

\bibitem {Dar}A. Y. Darani, F. Soheilnia, 2-absorbing and weakly 2-absorbing
submodules, Thai J. Math, 9(3) (2011), 577-584.

\bibitem {ES1988}Z. A. El-Bast and P. P. Smith, \textit{Multiplication
modules}, Communications in Algebra, 16(4) (1988), 755--779.

\bibitem {Rachida}R. El Khalfaoui, N. Mahdou, P. Sahandi, N. Shirmohammadi,
Amalgamated modules along an ideal. Communications of the Korean Mathematical
Society, \textbf{36(1) }(2021), 1-10.

\bibitem {ElKhalfi-Kim-Mahdou}A. El Khalfi, H. Kim, and N. Mahdou,
Amalgamation extension in commutative ring theory: a survey, \textit{Moroccan
Journal of Algebra and Geometry with Applications}, 1(1) (2022), 139--182.

\bibitem {Khashan2025}H. A. Khashan and E. Yetkin Celikel,
\textit{Square-difference factor absorbing submodules}, Analele
\c{S}tiin\c{t}ifice ale Universit\u{a}\c{t}ii Ovidius Constan\c{t}a, 33(3), 2025.

\bibitem {KhashanCelikel}H. A. Khashan, E. Yetkin Celikel, and \"{U}. Tekir,
\textit{Square-Difference Factor Absorbing Primary Ideals of Commutative
Rings}, Journal of Algebra and Its Applications, 2026, doi.org/10.1142/S0219498827500824

\bibitem {Lang}S. Lang, \textit{Algebra}, 3rd ed., Graduate Texts in
Mathematics, Vol. 211, Springer-Verlag, New York, 2002.

\bibitem {M1992}R. L. McCasland and M. E. Moore, \textit{On prime submodules},
Communications in Algebra, 20(11) (1992), 3403--3417.

\bibitem {M}H. Mostafanasab, E. Yetkin, \"{U}. Tekir, A. Y. Darani, On
2-absorbing primary submodules of modules over commutative rings. Analele
stiintifice ale Universitatii" Ovidius" Constanta. Seria Matematica, 24(1),
(2016), 335-351.

\bibitem {Pay}Sh. Payrovi, S. Babaei,  On 2-absorbing submodules, Algebra
Colloquium, 19(1), (2012), 913-920.

\bibitem {SYZ1994}H. Sharif, S. Yassemi, and R. Zaare-Nahandi, \textit{Primary
submodules}, Communications in Algebra, 22(13) (1994), 5239--5244.
\end{thebibliography}
\end{document}